\documentclass[11pt]{amsart}
\usepackage{amssymb}
\usepackage{amsthm}
\usepackage{amsmath}
\usepackage{latexsym}
\usepackage{comment}
\usepackage{hyperref}
\usepackage{verbatim}
\usepackage{xypic}
\usepackage{graphicx}
\usepackage[utf8]{inputenc}
\usepackage[margin=1in]{geometry}

\newtheorem{thm}{Theorem}[section]

\newtheorem{cor}[thm]{Corollary}

\theoremstyle{definition} \newtheorem{ex}[thm]{Example}

\theoremstyle{definition} \newtheorem{rmk}[thm]{Remark}

\newcommand{\qq}{\mathbb{Q}}

\newcommand{\sst}{\mathrm{ss}}

\graphicspath{{pics/}}

\title{Semistable models of elliptic curves over residue characteristic 2}
\author{Jeffrey Yelton}

\begin{document}

\maketitle

\begin{abstract}

Given an elliptic curve $E$ in Legendre form $y^2 = x(x - 1)(x - \lambda)$ over the fraction field of a Henselian ring $R$ of mixed characteristic $(0, 2)$, we present an algorithm for determining a semistable model of $E$ over $R$ which depends only on the valuation of $\lambda$.  We provide several examples along with an easy corollary concerning $2$-torsion.

\end{abstract}

\

Let $R$ be a Henselian ring of mixed characteristic $(0, 2)$ with a discrete valuation $v : K^{\times} \to \qq$ normalized so that $v(2) = 1$, and let $K$ be its fraction field.  Let $E$ be the elliptic curve over $K$ defined by an equation of the form $y^2 = f(x)$ for some separable polynomial $f(x) \in K[x]$.  After replacing $K$ by a suitable extension and possibly scaling $y$ by an element of $K$ to get an isomorphic elliptic curve, we assume that $\alpha_1, \alpha_2, \alpha_3 \in R$ with $\alpha_2 - \alpha_1 \in R^{\times}$.  After possibly applying another isomorphism which translates $x$ by $-\alpha_1$ and then scales it by $(\alpha_2 - \alpha_1)^{-1}$, we further assume that $E$ is in \textit{Legendre form}; that is, $E$ is a smooth projective model of an affine curve given by an equation of the form 
$$y^2 = f(x) := x(x - 1)(x - \lambda)$$
 with $\lambda \in R \smallsetminus \{0, 1\}$ (we denote the point at infinity by $\mathcal{O} \in E(K)$).  The purpose of this note is to explicitly find a semistable model of $E$ over a finite extension of $K$.  More precisely, we will find a finite extension $K' / K$ such that $E / K'$ has a model $E^{\sst}$ given by explicit formulas with coefficients in the ring of integers of $K'$ and which has either good or (split) multiplicative reduction.

It is well known (see for instance \cite[\S IV.1.2]{serre1989abelian} or \cite[Proposition VII.5.5]{silverman2009arithmetic}) that any elliptic curve over a discrete valuation field has good (resp. multiplicative) reduction over some finite algebraic extension of that field if and only if the valuation of its $j$-invariant is nonnegative (resp. negative).  The formula for the $j$-invariant of the Legendre curve $E$ is given as in \cite[Proposition III.1.7]{silverman2009arithmetic} by 
\begin{equation} \label{eq j} 
j(E) = 2^8 \frac{(\lambda^2 - \lambda + 1)^3}{\lambda^2 (\lambda - 1)^2}. 
\end{equation}
  For simplicity, we assume throughout this paper that $v(\lambda - 1) = 0$, noting that if $v(\lambda - 1) > 0$, then we have $v(\lambda) = 0$ and the assumption becomes true after replacing $\lambda$ by $1 - \lambda$ and applying the isomorphism given by $(x, y) \mapsto (1 - x, y)$.  It follows from this assumption and the formula in (\ref{eq j}) that $v(j(E)) = 8 - 2v(\lambda)$ and that therefore any semistable model $E^{\sst}$ has good (resp. multiplicative) reduction if and only if $v(\lambda) \leq 4$ (resp. $v(\lambda) > 4$).  This explains ``why" the formula for the $j$-invariant includes an ``extra" factor of $2^8$.  The equivalence between potential good reduction and integrality of the $j$-invariant over residue characteristic $2$ is derived by Silverman as \cite[Corollary A.1.4]{silverman2009arithmetic} by converting $E$ to its Deuring normal form and arguing via manipulations involving the $j$-invariant.  In the course of constructing a semistable model of $E$, we will show essentially the same result more directly and without invoking the $j$-invariant.

It appears that the first investigations of how to explicitly determine semistable models of curves which are Galois $p$-covers of projective lines over DVRs of residue characteristic $p$ were done by Coleman in \cite[\S6]{coleman1987computing}.  The general approach proposed by Coleman led to results of Lehr and Matignon in \cite{lehr2006wild}, where they investigate superelliptic curves given by equations of the form $y^p = f(x)$ under the assumption that the roots of $f$ are \textit{equidistant} (that is, after scaling $x$ and $y$ by appropriate powers of a uniformizer, the roots of $f$ specialize to distinct elements in the residue field).  The ``monodromy polynomial" defined in that paper, which is used to define an extenion of the DVR over which the superelliptic curve obtains semistable reduction, is essentially the polynomial $P(x)$ in the statement of Theorem \ref{thm m<4}.  However, the work of both Coleman and of Lehr-Matignon focus on finding \textit{stable} models of curves of genus greater than $1$, and although the construction in \cite{lehr2006wild} of ``stably marked" models using the monodromy polynomial also works for elliptic curves, the equidistance hypothesis assumed there translates to the restrictive assumption that $v(\lambda) = 0$ in the context of our paper.  Our goal is to describe a completely explicit method of finding semistable models of elliptic curves in Weierstrass form over mixed characteristic $(0, 2)$ which shall be presented in a more elementary fashion than the results in \cite{coleman1987computing} and \cite{lehr2006wild}.  To the best of the author's knowledge, such a method for general elliptic curves is not present in the literature, although particular examples are done in \cite[\S4.1]{bouw2015semistable} (indeed, some of the ideas and notation used in this note were inspired by \cite{bouw2015semistable}).  We believe that the strategy presented here is also applicable to determining semistable models and reduction types for hyperelliptic curves over mixed characteristic $(0, 2)$, as suggested in both \cite{coleman1987computing} and \cite{lehr2006wild}.

\subsection{Our general set-up} \label{S1}

Given a finite extension $R' / R$ of Henselian rings with fraction fields $K' / K$, by a \textit{Weierstrass model} of $E$ over $R'$ we mean an elliptic curve $E' / K'$ isomorphic over $K'$ to $E / K'$, which is determined by an equation of the form 
\begin{equation} \label{eq model} y^2 + a_1 xy + a_3 y = x^3 + a_2 x^2 + a_4 x + a_6, \end{equation}
 with all $a_i \in R'$.  We write $\bar{E}'$ for the reduction of $E'$; it is a projective curve over the residue field of $R'$.  If this curve is either smooth or has only a single node, we say that the Weierstrass model $E'$ is \textit{semistable}. By \cite[Proposition VII.5.4]{silverman2009arithmetic}, there is always a finite extension $K' / K$ and a semistable Weierstrass model $E^{\sst}$  of $E$ over $R'$.  We note that at least one of $a_1$ and $a_3$ must be a unit in $R$ to ensure that $E^{\sst}$ does not have a cusp; that $v(a_1) > 0$ is then sufficient to ensure smoothness of $\bar{E}^{\sst}$; and that $v(a_3) > 0$ on the other hand implies that $E^{\sst}$ has a node at $(0, 0)$.

An equation of the form given in (\ref{eq model}) can be converted to an equation of the form $y^2 = F(x) \in K'[x]$ for some finite extension $K' / K$ by completing the square: we replace $y$ by $y - \frac{1}{2}(a_1 x + a_3)$.  Then an isomorphism from the curve $E$ given by $y^2 = f(x)$ to the curve given by $y^2 = F(x)$ must be of the form $(x, y) \mapsto (\alpha + \beta x, \beta^{3/2} y)$ for some $\alpha, \beta \in K'$ (in fact, $K'$ will just be the extension given by adjoining the elements $\alpha$ and $\beta^{1/2}$ to $K$).  Given such elements $\alpha, \beta$, we first observe that this isomorphism maps $E$ to the curve defined by 
\begin{equation} \label{eq model F} y^2 = F(x) = F_{\alpha, \beta}(X) := (x + \alpha\beta^{-1})(x + \alpha\beta^{-1} - \beta^{-1})(x + \alpha\beta^{-1} - \lambda\beta^{-1}). \end{equation}
  Now we want to find polynomials $G(x) = x^3 + a_2 x^2 + a_4 x + a_6 \in R'[x]$ and $H(x) = a_1 x + a_3 \in R'[x]$ such that $F = G + \frac{1}{4}H^2$; then the isomorphism $(x, y) \mapsto (x, y + H(x))$ maps the curve given by (\ref{eq model F}) to the one given by (\ref{eq model}).

For each integer $n \geq 1$, we write $F^{(n)}$ for the $n$th derivative of $F$ divided by $n!$, so that $F^{(n)}(0)$ equals the coefficient of the $x^n$-term of $F$.  We compute formulas for the elements $a_1, a_2, a_3$, using the fact that 
\begin{equation} G(x) + \frac{1}{4}H(x)^2 = F(x) = x^3 + F^{(2)}(0) x^2 + F^{(1)}(0) x + F(0). \end{equation}
  Our formulas are as follows:
\begin{equation} \label{eq formulas} \begin{split} a_3 &= 2\sqrt{F(0) - a_6} \\ a_1 &= \frac{2F^{(1)}(0) - a_4}{a_3} = \frac{F^{(1)}(0) - a_4}{\sqrt{F(0) - a_6}} \\ a_2 &= F^{(2)}(0) - \frac{1}{4}a_1^2 = F^{(2)}(0) - \frac{(F^{(1)}(0) - a_4)^2}{4(F(0) - a_6)} \end{split} \end{equation}
It will be convenient to fix $a_4 = a_6 = 0$ so that the elements $a_1, a_2, a_3$ are completely determined (up to choosing a sign for $a_3$) by our choice of $\alpha$ and $\beta$ and are given by the slightly simpler formulas 
\begin{equation} \label{eq formulas2} a_1 = \frac{F^{(1)}(0)}{\sqrt{F(0)}}; \ a_2 = F^{(2)}(0) - \frac{1}{4}a_1^2 = F^{(2)}(0) - \frac{F^{(1)}(0)^2}{4F(0)}; \ a_3 = 2\sqrt{F(0)}. \end{equation}

\subsection{The $v(\lambda) < 4$ case} \label{S2}

In this section we assume that $v(j(E)) > 0$.  We then see from the formula in (\ref{eq j}) that we have $0 \leq m := v(\lambda) < 4$.  Since $j(E)$ is integral, the desired model $E^{\sst}$ should have good reduction.  For any $\lambda$ with $0 \leq v(\lambda) < 4$, we now show how to find algebraic elements $\alpha, \beta \in \bar{K}$ such that we get $v(a_1) > 0$, $v(a_3) = 0$ and even allow $a_2$ to be any integral element that we choose.

\begin{thm} \label{thm m<4}

Assume that $0 \leq m < 4$.  Choose any $\beta \in \bar{K}$ such that $v(\beta) = \frac{1}{3}m + \frac{2}{3}$ (where $v$ is extended uniquely to a discrete valuation on $K(\beta)$).  Let $\alpha$ be a root of the polynomial 
\begin{equation} \label{eq val<4} P(X) := 3X^4 - 4(1 + \lambda) X^3 + 6\lambda X^2 - \lambda^2 - \delta, \end{equation}
 where $\delta \in K(\beta^{1/2})$ satisfies $v(\delta) \geq \frac{4}{3}m + \frac{8}{3}$ (e.g. $\delta = 0$).  Then $E$ is isomorphic over $K' := K(\beta^{1/2}, \alpha, F(0)^{1/2})$ to the elliptic curve $E^{\sst}$ given by the equation in (\ref{eq model}), where $a_4 = a_6 = 0$ and the other coefficients $a_i$ are given by the formulas in (\ref{eq formulas2}).  The isomorphism $\varphi : E \stackrel{\sim}{\to} E^{\sst}$ is given by composing the map $(x, y) \mapsto (\alpha + \beta x, \beta^{3/2}y)$ with the map $(x, y) \mapsto (x, y + \frac{1}{2}(a_1 x + a_3))$.

We have $v(a_1) > 0$, $v(a_2) \geq 0$, and $v(a_3) = 0$ (which implies that $E^{\sst}$ has good reduction).  Moreover, we have 
$$a_2 = \frac{\delta}{4\beta \alpha(\alpha - 1)(\alpha - \lambda)}.$$

\end{thm}

\begin{proof}

Assume that we have chosen an algebraic element $\beta$ with $v(\beta) = \frac{1}{3}m + \frac{2}{3}$, an element $\delta \in K(\beta^{1/2})$ satisfying $v(\delta) = \frac{1}{3}m + \frac{2}{3}$, and a root $\alpha$ of the polynomial $P(X)$.  The first statement just reaffirms what was shown in the above discussion where the formulas for $a_1, a_2, a_3 \in K'$ were derived, so our main task is to demonstrate the desired bounds for the valuations of these elements.

We note first that $v(\lambda^2) = 2m < \frac{4}{3}m + \frac{8}{3}$, so that the constant coefficient of the polynomial $P$ has valuation equal to $2m$ regardless of our choice of $\delta$.  Then since the coefficient of $X^4$ is a unit and the coefficient of $X^3$ (resp. $X^2$) has valuation at least $2 \geq \frac{1}{2}m$ (resp. equal to $1 + m > m$), the Newton polygon of this polynomial consists of a single line segment with slope $\frac{1}{2}m$.  It follows that $v(\alpha) = \frac{1}{2}m$.  We clearly have $v(\alpha - 1) = 0$ and $v(\alpha - \lambda) = v(\alpha) = \frac{1}{2}m$ as long as $m > 0$.  If $m = 0$, we claim that these equalities still hold so that $v(\alpha - 1) = v(\alpha - \lambda) = 0$.  To see this, assume that $m = 0$ and consider the polynomials $P(X + 1)$ and $P(X + \lambda)$; it is straightforward to calculate (using the fact that $v(\lambda) = v(\lambda - 1) = 0$) that the Newton polygons of these shifted polynomials both coincide with the Newton polygon of $P$, and the claim follows from the fact that $\alpha - 1$ and $\alpha - \lambda$ are roots of the respective polynomials.  We now have 
\begin{equation} v(F(0)) = v(\alpha) + v(\alpha - 1) + v(\alpha - \lambda) - 3v(\beta) = 2v(\alpha) - 3v(\beta) = m - m - 2 = -2. \end{equation}
 The desired equality $v(a_3) = v(2\sqrt{F(0)}) = 0$ immediately follows.

Now we treat the requirement that $v(a_2) \geq 0$, using the formula for $b_2$ given in (\ref{eq formulas2}).  We use the formulas 
\begin{equation} \begin{split} \label{eq formulas3} &\beta^3 F(0) = \alpha(\alpha - 1)(\alpha - \lambda); \ \ \beta^2 F^{(1)}(0) = \alpha(\alpha - 1) + \alpha(\alpha - \lambda) + (\alpha - 1)(\alpha - \lambda); \\ 
&\beta F^{(2)}(0) = \alpha + (\alpha - 1) + (\alpha - \lambda) \end{split} \end{equation}
 to expand $4\beta^4 F(0)a_2 = 4(\beta F^{(2)}(0))(\beta^3 F(0)) - (\beta^2 F^{(1)}(0))^2$ as 
\begin{equation*} \begin{split} 4\alpha^2(\alpha - 1)(\alpha - \lambda) + &4\alpha(\alpha - 1)^2(\alpha - \lambda) + 4\alpha(\alpha - 1)(\alpha - \lambda)^2 \\ 
 - &(\alpha^2(\alpha - 1)^2 + \alpha^2(\alpha - \lambda)^2 + (\alpha - 1)^2(\alpha - \lambda)^2 \\ 
 &\hspace{2em} + 2\alpha^2(\alpha - 1)(\alpha - \lambda) + 2\alpha(\alpha - 1)^2(\alpha - \lambda) + 2\alpha(\alpha - 1)(\alpha - \lambda)^2) \end{split} \end{equation*}
\begin{equation*} \begin{split} &= 2\alpha^2(\alpha - 1)(\alpha - \lambda) + 2\alpha(\alpha - 1)^2(\alpha - \lambda) + 2\alpha(\alpha - 1)(\alpha - \lambda)^2 - \alpha^2(\alpha - 1)^2 - \alpha^2(\alpha - \lambda)^2 - (\alpha - 1)^2(\alpha - \lambda)^2 \\ 
 &= 2(2\alpha - \lambda)(\alpha)(\alpha - 1)(\alpha - \lambda) - \alpha^2(\alpha - \lambda)^2 + (\alpha - 1)^2 (2\alpha(\alpha - \lambda) - \alpha^2 - (\alpha - \lambda)^2) \\
 &= 2(2\alpha - \lambda)(\alpha)(\alpha - 1)(\alpha - \lambda) - \alpha^2(\alpha - \lambda)^2 - (\alpha - 1)^2\lambda^2 \\ 
 &= [4\alpha^4 - 2(2 + 3\lambda)\alpha^3 + 2(3\lambda + \lambda^2)\alpha^2 - 2\lambda^2\alpha] - [\alpha^4 - 2\lambda\alpha^3 + \lambda^2\alpha^2] - [\lambda^2\alpha^2 - 2\lambda^2\alpha + \lambda^2] \end{split} \end{equation*}
\begin{equation} \label{eq a_2} = 3\alpha^4 - 4(1 + \lambda)\alpha^3 + 6\lambda \alpha^2 - \lambda^2. \end{equation}
Thus, since $P(\alpha) = 0$ can be written as the above expression minus the element $\delta$, we have $\delta = 4\beta^4 F(0)a_2$ (implying the claimed formula for $a_2$).  Now the fact that $v(a_2) \geq 0$ is equivalent to saying that $v(\delta) \geq 2 + 4v(\beta) + v(F(0)) = \frac{1}{3}m + \frac{2}{3}$, which was indeed our condition for $\delta$.

It remains only to check that $v(a_1) > 0$.  Note that $v(\beta F^{(2)}(0)) \geq \min\{v(\alpha), v(1)\} = 0$.  It follows from the formula for $a_2$ in (\ref{eq formulas2}) that $v(\frac{1}{4}a_1) \geq \min\{v(F^{(2)}(0)), v(a_2)\} = v(\beta^{-1}) > -2$.  Therefore, $v(a_1) - 2 > -2$, hence the desired inequality.

\end{proof}

\begin{ex} \label{ex m=0}

Suppose we want to find a semistable model $E^{\sst}$ of the elliptic curve $E / \qq_2$ given by $y^2 = x^3 - 1$ at the prime $(2)$.  This elliptic curve is well known to be CM, and so any semistable model $E^{\sst}$ should have good reduction; we can also see this by noting that $j(E) = 0$.  In fact, $E$ is isomorphic (over $K := \qq_2(\omega)$) to the Legendre curve with $\lambda = -\omega^2$, where $\omega := \frac{1}{2}(-1 + \sqrt{-3})$ is a primitive cube root of unity; since $m = v(\lambda) = 0 < 4$, we may apply Theorem \ref{thm m<4}.

We have 
\begin{equation} P(X) = 3X^4 - 4(1 - \omega^2) X^3 - 6\omega^2 X^2 - \omega - \delta. \end{equation}
By an easy computation, plugging in $X = \omega$ to the above polynomial yields $8\omega^2 - \delta$, so we may take $\delta = 8\omega^2$ (noting that $v(\delta) = 3 \geq \frac{4}{3}m + \frac{8}{3}$) and $\alpha = \omega$.  Then we may choose $\beta$ to be any element with valuation $\frac{1}{3}m + \frac{2}{3} = \frac{2}{3}$, say $\beta = 2^{2/3}$.  Now evaluating the formulas in (\ref{eq formulas2}) yields the following equation for $E^{\sst}$ over the (abelian) extension $K' := K((-3)^{1/4}, 2^{1/3})$.
\begin{equation} y^2 - \omega(-3)^{-1/4}2^{5/3}xy + (-3)^{1/4}y = x^3 + \omega^2(-3)^{-1/2}2^{1/3} x^2 \end{equation}
  We see that $E$ and $E^{\sst}$ are isomorphic over $K'$ and that the reduction $\bar{E}^{\sst}$ is the nonsingular curve given by $y^2 + y = x^3$.

\end{ex}

\begin{ex} \label{ex m=1}

Suppose we want to find a semistable model $E^{\sst}$ of the elliptic curve $E / \qq_2$ given by $y^2 = x^3 - x$ at the prime $(2)$.  Just as in the previous example, this elliptic curve is CM, and so any semistable model $E^{\sst}$ should again have good reduction.  Moreover, $E$ is isomorphic over $\qq_2$ to the Legendre curve with $\lambda = 2$, and since $m = v(\lambda) = 1 < 4$, we may apply Theorem \ref{thm m<4}.

We let $\beta = 2$, noting that this choice of $\beta$ satisfies the requirement that $v(\beta) = \frac{1}{3}m + \frac{2}{3} = 1$.  Then we have 
\begin{equation} P(X) = 3X^4 - 12X^3 + 12X^2 - 4 - \delta. \end{equation}
One can readily check that if we set $\delta = 0$, the roots of this polynomial are $1 \pm \sqrt{1 \pm \frac{2}{\sqrt{3}}}$, where the choices of sign are independent.  We take $\alpha = 1 + \sqrt{1 + \frac{2}{\sqrt{3}}}$.  Now evaluating the formulas in (\ref{eq formulas2}) yields the following equation for $E^{\sst}$, over the extension $K' := \qq_2(2^{1/2}, 3^{1/4}, \sqrt{\sqrt{3} + 2})$ (which is abelian over $\qq_2(i)$ as it is contained in $\qq_2(\zeta_{24}, 3^{1/4})$, where $\zeta_{24}$ is a primitive $24$th root of unity).
\begin{equation} y^2 + (3^{1/4} + 3^{3/4})(1 + \frac{2}{\sqrt{3}})^{-1/4} xy + 3^{-1/4}(1 + \frac{2}{\sqrt{3}})^{1/4} y = x^3 \end{equation}
  We see that $E$ and $E^{\sst}$ are isomorphic over $K'$ and that the reduction $\bar{E}^{\sst}$ is again the nonsingular curve given by $y^2 + y = x^3$.

\end{ex}

\subsection{The $v(\lambda) \geq 4$ case} \label{S3}

For this section, we adopt exactly the same set-up but treat the complimentary case where $v(j(E)) \leq 0$.  In this case, we see from the formula in (\ref{eq j}) that we have $m := v(\lambda) \geq 4$.  Therefore, under this assumption, any semistable model $E^{\sst}$ should have good reduction if and only if $m = 4$; otherwise $E^{\sst}$ has multiplicative reduction.  As in \S\ref{S2}, we will show how to find algebraic elements $\alpha, \beta \in \bar{K}$ such that evaluating $a_1, a_2, a_3 \in K' := K(\beta^{1/2}, \alpha, F(0)^{1/2})$ using the formulas in (\ref{eq formulas2}) yields an equation of the form in (\ref{eq model}) (with $a_4 = a_6 = 0$) for an elliptic curve with semistable reduction.

\begin{thm} \label{thm m>4}

Assume that $m \geq 4$.  Let $\beta \in (K^{\times})^2$ be any element such that $v(\beta) = 2$ (e.g. $\beta = 4$), and choose an element $\alpha \in K$ such that $2 \leq v(\alpha) \leq m - 2$.  Then $E$ is isomorphic over $K' := K(F(0)^{1/2})$ to the elliptic curve $E^{\sst}$ given by the equation in (\ref{eq model}), where $a_4 = a_6 = 0$ and the other coefficients $a_i$ are given by the formulas in (\ref{eq formulas2}).

We have $v(a_1) = 0$, $v(a_2) \geq 0$, and $v(a_3) = v(\alpha) - 2$ (when $v(\alpha) > 2$, this directly implies that $E^{\sst}$ has multiplicative reduction).  The curve $E^{\sst}$ has good reduction if $m = 4$ and has multiplicative reduction otherwise.

\end{thm}

\begin{proof}

First of all, we note that $v(\alpha - 1) = 0$.  The condition that $v(\alpha) \leq m - 2$ ensures that $v(\lambda) > v(\alpha)$, so $v(\alpha - \lambda) = v(\alpha)$.  Therefore, we have 
\begin{equation} v(F(0)) = v(\alpha) + v(\alpha - 1) + v(\alpha - \lambda) - 3v(\beta) = 2v(\alpha) - 6; \end{equation}
 and 
\begin{equation} v(F^{(1)}(0)) = v(2\alpha(\alpha - 1) - \lambda(\alpha - 1) + \alpha(\alpha - \lambda)) - 2v(\beta) = \min\{v(\alpha) +1, m, 2v(\alpha)\} - 4 = v(\alpha) - 3. \end{equation}
  It follows that $v(a_3) = v(\alpha) - 2 \geq 0$ and $v(a_1) = 0$; in particular, $v(a_3) = 0$ if and only if $m = 4$.

We next check that $v(a_2) \geq 0$.  In order to do so, we recall the formula in (\ref{eq a_2}) which we derived earlier: 
\begin{equation} \label{eq a_2 m>4} 4\beta^4 F(0)a_2 = 3\alpha^4 - 4(1 + \lambda)\alpha^3 + 6\lambda \alpha^2 - \lambda^2. \end{equation}
  Since $\min\{v(3\alpha^4), v(4(1 + \lambda)\alpha^3), v(6\lambda\alpha^2), v(\lambda^2)\} = \min\{4v(\alpha), 3v(\alpha)+ 2, 2v(\alpha) + m + 1, 2m\} \geq 2v(\alpha) + 4$, we have $v(a_2) = v(4F(0)\beta^4 a_2) - v(4F(0)) - 4v(\beta) \geq 2v(\alpha) + 4 - (2v(\alpha) - 4) - 8 = 0$, as desired.

Finally we assume that $v(\alpha) = 2$ and set out to show that the curve $E^{\sst}$ has good reduction if and only if $m = 4$.  Any singular point $(x, y)$ on $\bar{E}^{\sst}$ satisfies the following set of equations.
\begin{equation} \begin{split} y^2 + \bar{a}_1 xy + \bar{a}_3 y &= x^3 + \bar{a}_2 x^2 \\ 
\bar{a}_1 y &= x^2 \\ 
\bar{a}_1 x + \bar{a}_3 &= 0 \end{split} \end{equation}
By solving for $x$ and $y$ in the bottom two equations and plugging the results in the top equation, we see that if such a point $(x, y)$ exists, we must have 
\begin{equation} \label{eq nonsingularity} \frac{\bar{a}_3^4}{\bar{a}_1^6} + \frac{\bar{a}_3^3}{\bar{a}_1^3} - \frac{\bar{a}_2\bar{a}_3^2}{\bar{a}_1^2} = 0
\end{equation}
 which, after dividing by $\frac{\bar{a}_3^2}{\bar{a}_1^2}$ and simplifying, yields 
\begin{equation}
\frac{\bar{a}_3}{\bar{a}_1}\Big(\frac{\bar{a}_3}{\bar{a}_1^3} + 1 \Big) - \bar{a}_2 = 0.
\end{equation}
  We now show that this is the case if and only if $m > 4$.  We compute the following equivalences modulo the prime ideal of $R'$, using the formulas in (\ref{eq formulas3}).
\begin{equation} \label{eq nonsingularity2} \begin{split} \frac{a_3}{a_1} = \frac{4F(0)}{2F^{(1)}(0)} &\equiv \frac{-4\alpha^2\beta^{-3}}{-4\alpha\beta^{-2}} = \frac{\alpha}{\beta} \\ 
\frac{a_3}{a_1^3} = \Big(\frac{a_3}{a_1}\Big)\Big(\frac{4F(0)}{4F^{(1)}(0)^2}\Big) &\equiv \Big(\frac{\alpha}{\beta}\Big) \frac{-4\alpha^2\beta^{-3}}{16\alpha^2\beta^{-4}} = -\frac{\alpha}{4} \end{split}\end{equation}
  Meanwhile, using what we know from (\ref{eq a_2 m>4}), we compute the equivalence 
\begin{equation} \label{eq nonsingularity3} a_2 = \frac{\beta^{-4}(3\alpha^4 - 4(1 + \lambda)\alpha^3 + 6\lambda \alpha^2 - \lambda^2)}{4F(0)} \equiv \frac{\beta^{-4}(\alpha^4 - 4\alpha^3 - \lambda^2)}{-4\alpha^2\beta^{-3}} = -\frac{\alpha^2}{4\beta} + \frac{\alpha}{\beta} + \frac{\lambda^2}{4\alpha^2\beta}. \end{equation}
  Putting (\ref{eq nonsingularity2}) and (\ref{eq nonsingularity3}) together, we get 
\begin{equation} \frac{a_3}{a_1}\Big(\frac{a_3}{a_1^3} + 1\Big) - a_2 \equiv \frac{\alpha}{\beta} \Big(-\frac{\alpha}{4} + 1\Big) + \frac{\alpha^2}{4\beta} - \frac{\alpha}{\beta} - \frac{\lambda^2}{4\alpha^2\beta} = -\frac{\lambda^2}{4\alpha^2\beta}. \end{equation}
  Since the valuation of the right-hand term is $2m - 2 - 4 - 2 = 2m - 8$, the above expression is equivalent to $0$ if and only if $m > 4$, and we are done.

\end{proof}

\begin{rmk}

It was pointed out to the author by Leonardo Fiore that in the situation of Theorem \ref{thm m>4}, a semistable model can be obtained by choosing $\alpha$ to be any element satisfying $v(\alpha) \geq 2$ (e.g. $\alpha = 0$), as long as we allow the possibility that $a_4 \neq 0$ or $a_6 \neq 0$.  Indeed, there is an isomorphism (defined over $R$) between any two such models induced by translating $x$ by the integral element $\beta^{-1}(\alpha_1 - \alpha_2) \in R$, where $\alpha_1$ and $\alpha_2$ are the choices of $\alpha$ determining the models.

\end{rmk}

\

We now recall that the $2$-torsion subgroup $E[2] \subset E(\bar{K})$ is given by $\{\mathcal{O}, (0, 0), (1, 0), (\lambda, 0)\}$. 

\begin{cor} \label{cor 2-torsion}

Assume that $m > 4$ and construct the semistable model $E^{\sst}$ of $E$ as in the statement of Theorem \ref{thm m>4}. The reduction of the $2$-torsion subgroup $E^{\sst}[2]$ coincides with the subset consisting of the infinity point $\bar{\mathcal{O}}$ and the cusp $P$ of $\bar{E}^{\sst}$; the inverse images of $\{\bar{\mathcal{O}}\}$ and $\{P\}$ correspond to the subgroup $\{\mathcal{O}, (1, 0)\} \subset E[2]$ and its coset $\{(0, 0), (\lambda, 0)\} \subset E[2]$ respectively.

\end{cor}

\begin{proof}

It is clear that the infinity point $\mathcal{O}$ of $E$ gets sent to $\bar{\mathcal{O}} \in \bar{E}^{\sst}$.  Now since $\varphi : E \stackrel{\sim}{\to} E^{\sst}$ sends the first coordinate of any point $(x, y) \in E(K') \smallsetminus \{\mathcal{O}\}$ to $\beta^{-1}(x - \alpha)$, we see that the first coordinate of the image $\varphi((1, 0)) \in E^{\sst}(K')$ (resp. of each image $\varphi((0, 0)), \varphi((\lambda, 0)) \in E^{\sst}(K')$) reduces to $\infty$ (resp. $-\frac{\bar{\alpha}}{\bar{\beta}}$).  As in the proof of Theorem \ref{thm m>4}, the cusp $P$ has $x$-coordinate $-\frac{\bar{a}_3}{\bar{a}_1} = -\frac{\bar{\alpha}}{\bar{\beta}}$.  Since $\bar{\mathcal{O}}$ (resp. $P$) is the only point of $\bar{E}^{\sst}$ whose first coordinate is $\infty$ (resp. $-\frac{\bar{\alpha}}{\bar{\beta}}$), we are done.

\end{proof}

\begin{rmk}

In a similar fashion to how we proved the above corollary, it is straightforward to show directly from Theorem \ref{thm m>4} (resp. Theorem \ref{thm m<4}) that in the case that $v(\lambda) = 4$ (resp. $0 \leq v(\lambda) < 4$), the elements $\mathcal{O}, (1, 0) \in E[2]$ are mapped via $\varphi$ composed with reduction to the infinity point $\bar{\mathcal{O}}$ of $\bar{E}^{\sst}$ and the elements $(0, 0), (\lambda, 0) \in E[2]$ map to another point of $\bar{E}^{\sst}$ (resp. the elements of $E[2]$ are all mapped to the infinity point $\bar{\mathcal{O}}$ of $\bar{E}^{\sst}$).  Since the image of $E[2]$ under $\varphi$ composed with reduction must be contained in the $2$-torsion subgroup $\bar{E}^{\sst}[2]$, we see in this way that when $v(\lambda) = 4$ (or equivalently, when $j(\bar{E}^{\sst}) \neq 0$), the reduced curve $\bar{E}^{\sst}$ is ordinary.  This is one direction of the equivalence given in \cite[Exercise 5.7]{silverman2009arithmetic}, which states that an elliptic curve over a field of characteristic $2$ is supersingular if and only if its $j$-invariant is $0$.  Since the other direction of that equivalence implies that $\bar{E}^{\sst}$ is supersingular in the $v(\lambda) < 4$ case, we see that $\bar{E}^{\sst}[2] = \{\bar{\mathcal{O}}\}$ coincides with the reduction of $E[2] \cong E^{\sst}[2]$.

\end{rmk}

\begin{ex} \label{ex m=4}

Consider the elliptic curve $E / \qq_2$ given by $y^2 = x(x - 1)(x - 16)$.  This curve is already in Legendre form with $\lambda = 16$.  Since $m = v(\lambda) = 4$, any semistable model $E^{\sst}$ will have good reduction, and we may apply Theorem \ref{thm m>4}.

We set $\alpha = \beta = 4$, noting that $v(\alpha) = 2 = m - 2$.  Now evaluating the formulas in (\ref{eq formulas2}) yields the following equation for $E^{\sst}$, over the extension $K' := \qq_2(i)$.
\begin{equation} y^2 + 3ixy + 3iy = x^3 + x^2 \end{equation}
  It is easy to check directly that the reduction $\bar{E}^{\sst}$, given by $y^2 + xy + y = x^3 + x^2$, is nonsingular.

\end{ex}

\begin{ex} \label{ex m=6}

Consider the elliptic curve $E / \qq_2$ given by $y^2 = x(x - 1)(x - 64)$.  The curve is again already in Legendre form, this time with $\lambda = 64$ so $m = v(\lambda) = 6$.  Again, we may apply Theorem \ref{thm m>4}, but in this case, the semistable model $E^{\sst}$ we arrive at will have multiplicative reduction.

As before, we set $\beta = 4$, but this time, we let $\alpha = 8$, noting that $2 < v(\alpha) = 3 \leq m - 4$.  Now evaluating the formulas in (\ref{eq formulas2}) yields the following equation for $E^{\sst}$, over the extension $K' := \qq_2(i)$.
\begin{equation} y^2 + 7ixy + 14iy = x^3 + 2x^2 \end{equation}
The reduction $\bar{E}^{\sst}$ is $y^2 + xy = x^3$, which visibly has a node at the point $(0, 0)$; hence, $E^{\sst}$ has (split) multiplicative reduction, as expected.

\end{ex}

\subsection{Acknowledgements}

The author is grateful to the referee, whose suggestions have helped to improve the exposition and to place this work in a broader context.

\bibliographystyle{plain}
\bibliography{bibfile}

\end{document}